\documentclass[11pt]{amsart}
\usepackage{geometry} 
\usepackage{hyperref}
\usepackage{bbm}
\usepackage{xcolor}
\usepackage{comment}
\usepackage{amssymb,amsthm,amsmath,amsfonts,amsbsy,latexsym,dsfont}
\usepackage{enumerate}   

\newtheorem{prop}{Proposition}
\newtheorem{thm}{Theorem}
\newtheorem{lemma}{Lemma}
\newtheorem{assume}{Assumption}
\newtheorem{remark}{Remark}
\newtheorem{definition}{Definition}
\newtheorem{cor}{Corollary}
\newtheorem{Ex}{Example}

\numberwithin{prop}{section}
\numberwithin{thm}{section}
\numberwithin{lemma}{section}
\numberwithin{assume}{section}
\numberwithin{remark}{section}
\numberwithin{equation}{section}
\numberwithin{definition}{section}
\numberwithin{cor}{section}
\numberwithin{Ex}{section}

\DeclareMathOperator*{\argmin}{argmin}

\def \E {\mathbb{E}}

\def \R {\mathbb{R}}

\def \cX{\mathcal{X}}
\def \cY {\mathcal{Y}}
\def \cZ{\mathcal{Z}}
\def \cE{\mathcal{E}}
\def \cP {\mathcal{P}}
\def \cF {\mathcal{F}}
\def \cL {\mathcal{L}}
\def \cV {\mathcal{V}}
\def \cW {\mathcal{W}}

\def \cT {\mathcal{T}}

\def \e {\epsilon}
\def \h {\hat}

\def \ol {\overline}
\def \pa {\partial}

\newcommand{\indep}{\perp \!\!\! \perp}

\title{Dynamic Cournot-Nash Equilibrium: The Non-Potential Case}
\author{Julio Backhoff-Veraguas}
\address{Department of Mathematics, University of Vienna}
\email{julio.backhoff@univie.ac.at}
\author{Xin Zhang} 
\address{Department of Mathematics, University of Vienna}
\email{xin.zhang@univie.ac.at}
\date{\today}

\begin{document}
\maketitle

\begin{abstract}

We consider a large population dynamic game in discrete time where players are characterized by time-evolving types. It is a natural assumption that the players' actions cannot  anticipate future values of their types. Such games go under the name of \emph{dynamic Cournot-Nash equilibria}, and were first studied by Acciaio et al.\ in \cite{AcBaJi21}, as a time/information dependent version of the games devised by  Blanchet and Carlier~\cite{BC16} for the static situation, under an extra assumption that the game is of potential type. The latter means that the game can be reduced to the resolution of an auxiliary variational problem.

In the present work we study dynamic Cournot-Nash equilibria in their natural generality, namely going beyond the potential case. As a first result, we derive existence and uniqueness of equilibria under suitable assumptions. Second, we study the convergence of the natural fixed-point iterations scheme in the quadratic case. Finally we illustrate the previously mentioned results in a toy model of optimal liquidation with price impact, which is a game of non-potential kind.
 
 \end{abstract}
 

\section{Introduction}

In this paper we consider a discrete-time dynamic game of mean field type. In this game, a representative player takes actions in time so as to minimize a cost functional which depends on her type, her action, and the distribution of actions of the whole population of players. Crucially, players' types may encode different characteristics or preferences, and may change progressively in time. The players' actions on a given date are only allowed to depend on their types up to that date, introducing an adaptability, or non-anticipativity, constraint into the game. The solutions to this game are dubbed \emph{dynamic Cournot-Nash equilibria} following Acciaio et al.\ \cite{AcBaJi21}. As in mean field games, searching for equilibria in dynamic Cournot-Nash games boils down to solving a fixed point problem, and an equilibrium to these games allows to build approximate equilibria in related large population symmetric games.

Building on the work \cite{BC16} by Blanchet and Carlier, it was shown in \cite{AcBaJi21} that the emerging field of causal optimal transport provides the right framework to describe dynamic Cournot-Nash games. However, when it comes to establishing existence or uniqueness of equilibria, the aforementioned paper makes the crucial assumption of the game being of \emph{potential type}. In a nutshell, this amounts to a structural assumption under which equilibria correspond to  minimizers of an auxiliary variational problem. However the assumption of being potential type is not ideal for multiple reasons. First, there are commonly used games/models of non-potential structure. Second, the link between causal optimal transport and dynamic Cournot-Nash games is blurred when one superimposes such structural assumption. Finally, the proposed method in \cite{AcBaJi21} was not only restricted to the potential case, but also a further cost-separability assumption was made, namely that the type of a player does not interact with the distribution of actions within the cost function. The goal of the present paper is to remedy these shortcomings, following the blueprint set forth in \cite{BC14b}, by Blanchet and Carlier, for the static case. 

We now summarize our contributions in some details.

In Section 2  we define the problem, recall the connection and the elements of causal optimal transport, and study the question of existence of (mixed) Nash equilibria. As customary, this is done by considering the best-response correspondence, which in our case assigns to any prior distribution ${\nu}$ of actions for the population of players the set $\Phi({\nu})$ of  optimal responses by a single player. Using causal transport, we establish the closedness and convexity of the set $\Phi({\nu})$. Applying Kakutani fixed point theorem, we obtain the existence of equilibria in our games under suitable assumptions. Finally, a uniqueness result is derived from a Lasry-Lions monotonicity condition. 

In Section 3 we assume a specific structure of the cost functional of the game, which allows us to find the equilibrium using the contraction mapping theorem. To do so, we use the structure of the game in order to get a hold on the best response correspondence. To this goal we use the fact that, conditioning on the past evolution of types,  the optimal response  can be constructed backwards (i.e.\ recursively) in time. Under appropriate Lipschitz and convexity assumptions, we prove that the best response is a contraction. 

In Section 4, we  introduce and study a simple optimal liquidation  problem in a price impact model. We first describe this model, and then establish the applicability of the results of Section 3. We prove that the game is not of potential type, and hence cannot be covered by the existing literature. Furthermore, we provide an example which illustrates how to compute the optimal response map and equilibrium. 

We close this introduction by giving a broader overview of the related literature.

\subsection{Related Literature}

The games we are concerned with are closely related to mean field games (MFG) in a discrete-time setting (see e.g.\ Gomes et al.\ \cite{GMS10}). For this parallel, the different types of agents considered in our setup correspond to different subpopulations of players in the MFG. 
The theory of mean field games aims at studying dynamic games as the number of agents tends to infinity. It was established independently by Lasry and Lions \cite{lasry2006jeux,LL07} and  by Huang, Malham\'e and Caines \cite{CHM06,huang2007large},
and has since seen a burst in activity, as e.g.\ documented in the monograph by Carmona and Delarue \cite{CD18}. See Cardialaguet's notes \cite{Car10}, based on P.L. Lions' lectures at Coll\'ege de France, for seminal results on mean field games, and also Bayraktar et al.\ \cite{BAYRAKTAR202198,MR3860894,MR4127851} or Cecchin and Fischer \cite{MR4082352} for the study of finite state mean field games.
The key assumption is that players are symmetric and weakly interacting through their empirical distributions, and the idea is to approximate large $N$-player systems by studying the behaviour as $N \rightarrow \infty$.

On the other hand, the notion of Cournot-Nash games has been pioneered by Blanchet and Carlier \cite{BC14a,BC16} who, building on the seminal contribution of Mas-Colell \cite{M84}, developed a connection between static Cournot-Nash equilibria and optimal transport. From a probabilistic perspective, large static anonymous games have been studied by Lacker and Ramanan in \cite{LR19}, with an emphasis on large deviations and the asymptotic behaviour of the Price of Anarchy. We also refer to this paper for a thorough review on the (vast) game theoretic literature. 
Building from this body of work, Acciaio et al.\ introduced in \cite{AcBaJi21} the concept of \emph{dynamic Cournot-Nash} game/equilibria. Working in the so-called potential case, that article studied questions of existence, convergence from finite to infinite populations, and computational aspects. Crucially, the article observed that instead of optimal transport, it is the theory of \emph{causal} optimal transport, which we discuss in the next paragraph, that plays the main role in the mathematical analysis of these games. Another article that took a similar, variational point of view is  \cite{benamou2017variational} wherein competitive games with mean field effect were studied. The advantage of the potential / variational setting, is that instead of studying an equilibrium problem, an auxiliary optimization problem is solved, which is in many ways better suited for analysis and computational resolution. To the best of our knowledge, the only article where non-potential (with non-separable costs) static Cournot-Nash games have been studied is Blanchet and Carlier's \cite{BC14b}. That article serves us as inspiration as we carry out our analysis of the dynamic case in similar non-potential settings.
 
As already mentioned, to deal with our dynamic setting, it is the tools from causal optimal transport (COT) rather than classical optimal transport that play a role. In a nutshell, COT is a relative of the optimal transport problem where an extra constraint, which takes into account the arrow of time (filtrations), is added. This in turn is crucial to ensure, in our application, the adaptedness of players' actions to their types in a dynamic framework.
The theory of COT, used to reformulate our asymptotic equilibrium problem, has been developed in the works \cite{BBLZ,La18}. This theory has been successfully employed in various applications, e.g.\ in mathematical finance and stochastic analysis \cite{ABC19,ABZ,BBBE19,2021arXiv210414245B}, in operations research \cite{Pflug,PflugPichler,PflugPichlerbook}, and in machine learning \cite{AMX}.

\vspace{10pt}

\noindent {\bf{Notation}.} Let $\cX_1, \dotso, \cX_N, \cY_1, \dotso, \cY_N$ be polish spaces, and take $\cX:=\cX_1 \times \dotso \times \cX_N, \cY:=\cY_1 \times \dotso \times \cY_N$. Define $\cX_{s:t}= \cX_s \times \dotso \times \cX_t$ and $\cY_{s:t}= \cY_s \times \dotso \times \cY_t$ for  $1 \leq s \leq t \leq N$. For  $x \in \cX$, we denote $x_{s:t}=(x_s, \dotso, x_t)$ for  $1 \leq s \leq t \leq N$, and similarly define $y_{s:t}$ for $y \in \cY$.  Denote the canonical filtration on $\cX$ and $\cY$ by $(\cF^{\cX}_t)_{t=1}^N$ and $(\cF^{\cY}_t)_{t=1}^N$ respectively.  For any polish space $\cZ$, we denote by $\cP(\cZ)$ the space of Borel probability measures on $\cZ$. Given $\eta \in \cP(\cX)$, and $\nu \in \cP(\cY)$, we denote the set of all couplings between $\eta$ and $\nu$ by 
$$
\Pi(\eta, \nu):=\{\pi \in \cP(\cX \times \cY): \, \pi(A \times \cY)=\eta(A), \, \pi(\cX \times B)=\nu(B), \, \forall \, A \in \cF^{\cX}_N, \, B \in \cF^{\cY}_N \}.
$$
The letter $\mathcal L$ stands for \emph{Law} and if $T:\cX\to\cY$ is measurable we denote by $T(\eta):=\eta\circ T^{-1}\in  \cP(\cY)$ the push-forward of $\eta$ by $T$.

\section{Existence by Set-Valued Fixed Point Theorem}
In this section, we formulate the Cournot-Nash equilibrium as a fixed point problem, and solve it by applying Kakutani fixed point theorem. First we recall the notion of causal coupling.
\begin{definition}
Suppose $\eta \in \cP(\cX), \, \nu \in \cP(\cY)$. A coupling $\pi \in \Pi(\eta, \nu)$ is said to be casual if  
\begin{align*}
\cF^{\cY}_{t} \underset{\cF^{\cX}_{t}}{\indep} \cF^{\cX}_{N}, \quad t=1,\dotso, N. 
\end{align*}
Denote by $\Pi_c(\eta, \nu)$ the collection of all causal couplings from $\eta$ to $\nu$. 
\end{definition}

\begin{remark}
In words, the above means that $\cF^{\cY}_{t}$ and $\cF^{\cX}_{N}$ are conditionally independent under $\pi$ given the information in $\cF^{\cX}_{t}$, and this for each $t$. See \cite{BBLZ,La18} for equivalent formulations of this condition, or our proof of Lemma \ref{lem:causality} below. The set $\Pi_c(\eta, \nu)$ is never empty, as the product of  $\eta$ and $\nu$ is always an element thereof. It is intructive to consider the case when $\pi$ is supported on the graph of a function $T$ from $\cX$ to $\cY$: in this case causality essentially boils down to the named function being adapted ($T(x)=(T_1(x_1),T_2(x_{1:2}),\dots , T_N(x_{1:N})$).
\end{remark}
In the rest of this paper, $N$ stands for a fixed time horizon. At each time $t \in \{1, \dotso, N\}$, a representative player is characterized by her type at that time, denoted by $x_t\in \cX_t$, and her control/action undertaken at that time, denoted by $y_t\in \cY_t$. Hence $x\in\cX$ and $y\in\cY$ denote the type-path and action-path of a player.  We fix once and for all $\eta \in \cP(\cX)$. The measure $\eta$ is the distribution of the types in the population of players, and is known in advance by the players.  

We denote $$\Pi_c(\eta, \cdot)=\{\Pi_c(\eta,\nu): \, \nu \in \cP(\cY) \} .$$
We now recall the notion of dynamic Cournout-Nash equilibrium (see \cite{AcBaJi21}), which we will simply call equilibrium in the rest of the work.
\begin{definition}
An equilibrium is a solution to the following fixed point problem 
\begin{align}\label{primalproblem} 
(i) \ \ &\h{\pi} \in \underset{\pi \in \Pi_c(\eta, \cdot)}{\argmin} \int_{\cX \times \cY} F(x,y,\h{\nu}) \, \pi(dx,dy) \text{ for some } \h{\nu} \in \cP(\cY), \\
(ii) \ \ & \text{The $\cY$-marginal of $\h{\pi}$  is $\h{\nu}$.} \notag
\end{align}
\end{definition}
Above $F: \cX \times \cY \times \cP(\cY) \to \R$ is a given cost function, assumed lower-bounded for the time being. Here $ \h{\nu}$ represents the distribution of controls/actions by the population of players, which is only determined at equilibrium, and $\h{\pi}$ characterizes the optimal response of each type of player given the cost function that they face $(x,y) \mapsto F(x,y, \h{\nu})$.

\begin{remark}
The above should be interpreted as randomized, or mixed strategies, equilibrium. A pure equilibrium would be an \emph{adapted} map $\hat T:\cX\to\cY$ satisfying 
\begin{align}\notag
(i') & \int_{\cX } F(x,\hat T(x),\h{\nu}) \, \eta(dx) = \inf_{T\,\text{adapted}} \int_{\cX } F(x,T(x),\h{\nu}) \, \eta(dx)  \text{ for some } \h{\nu} \in \cP(\cY), \\
(ii') & \,\, \hat T(\eta)=\h{\nu}, \text{i.e. the image of $\eta$ by $\hat T$ is $\h{\nu}$}. \notag
\end{align}

\end{remark}

As usual in game theory we introduce the best-response set-valued map, or correspondence, defined by 
\begin{align}\label{eq:decomposition}
\Phi(\nu):=\left\{\pi \in \Pi_c(\eta, \cdot): \, \int F(x,y,\nu) \, \pi(dx ,dy) \leq \int F(x,y,\nu) \, \pi'(dx ,dy), \forall  \pi' \in \Pi_c(\eta, \cdot) \right\},
\end{align}
and also the projection from $\Pi_c(\eta, \cdot)$ to $\cP(\cY)$
\begin{align*}
Pj:   \pi  \mapsto \text{$\cY$-marginal of  $\pi$}.
\end{align*}
Finally we introduce $$T(\h{\nu}):=Pj \circ \Phi (\h{\nu}),$$
the $\cY$-marginals of the best responses to $\h{\nu}$, i.e.\ the possible distributions of actions in response to $\h{\nu}$. 

It can be readily seen that $\h{\nu}$ is a fixed point as in \eqref{primalproblem} if and only if $\h{\nu} \in T(\h{\nu})$. We will show the existence of fixed points of $T$ applying Kakutani fixed point theorem, which we recall in the following lemma. 

\begin{lemma}\label{lem:fixedpoint}
Let $R: \cZ \to 2^{\cZ}$ be a set-valued map. Then $R$ has a fixed point, i.e. $\exists z$ s.t. $z \in R(z)$,  if 
\begin{enumerate}[(i)]
\item $\cZ$ is a nonempty compact, convex set in a locally convex space. 
\item $R$ is upper semi-continuous, and the set $R(y)$ is nonempty, closed, and convex for all $z \in \cZ$. 
\end{enumerate}
\end{lemma}
\begin{proof}
See \cite[Theorem 9.B]{ZeidlerEberhard1986Nfaa}. 
\end{proof}

The following lemma will be used to show that $T(\nu)$ is closed and convex for any $\nu \in \cP(\cY)$. See \cite{BBLZ,La18} for similar statements: We present it here, separately, for the sake of clarity.

\begin{lemma}\label{lem:causality}
Causality is preserved under weak convergence, i.e., $\pi \in \Pi_c(\eta, \cdot)$ if $\pi=\lim\limits_{n \to \infty} \pi_n$ for a sequence $(\pi_n)_{n \geq 0} \subset \Pi_c(\eta, \cdot)$, and so $\Pi_c(\eta, \cdot)$ is closed. Also $\Pi_c(\eta, \cdot)$ is convex, i.e., $a\pi_1+(1-a)\pi_2 \in \Pi_c(\eta, \cdot)$ for any $ \pi_1, \pi_2 \in \Pi_c(\eta, \cdot)$ and $a \in [0,1]$. 
\end{lemma}
\begin{proof} Clearly the $\cX$-marginal of $\pi$ is $\eta$. 
Let us prove that $\cF^{\cY}_{t} \underset{\cF^{\cX}_{t}}{\indep} \cF^{\cX}_{N}$ under $\pi$ for any $t \in \{1, \dotso, N\}$. This is equivalent to proving that, for any bounded continuous function $g: \cY_{1:t} \to \R$,  it holds  
\begin{align*}
\E^{\pi} \left[g(Y_{1:t}) \, | \, \cF^{\cX}_t\right]= \E^{\pi} \left[g(Y_{1:t}) \, | \, \cF^{\cX}_N\right],
\end{align*}
where $Y_{1:t}: \cY \to \cY_{1:t}$ is the projection map on the first $t$ coordinates. Denote by $\eta_{x_{1:t}}(dx_{t+1:N})$ the disintegration of $\eta$ on the first $t$ components $x_{1:t}$. Then it suffices to prove that 
\begin{align}\label{eq:causality1}
&\int_{\cX \times \cY} g(y_{1:t})  f(x)\, \pi(dx, dy) \notag \\
&=\int_{\cX_{1:t} \times \cY_{1:t}} g(y_{1:t})  \left(\int_{\cX_{t+1:N}} f(x_{1:t}, x_{t+1:N}) \, \eta_{x_{1:t}}(dx_{t+1:N}) \right)\, \pi(dx_{1:t} dy_{1:t}),
\end{align}
for any bounded continuous function $f: \cX \to \R$.  Since the function 
\begin{align*}
\bar{f}(x_{1:t}):= \int_{\cX_{t+1:N}} f(x_{1:t}, x_{t+1:N}) \, \eta_{x_{1:t}}(dx_{t+1:N})
\end{align*}
 is measurable, by Lusin's Theorem, there exists a closed $\cV \subset \cX_{1:t}$ such that $\eta(\cV) > 1-\delta$ and $\bar{f}$ is continuous restricted to $\cV$. Then by Tietze's Theorem, we extend $\bar{f}$ to a bounded continuous function $\bar{f}'$ on $\cX_{1:t}$, and it is clear that $f|_{\cV}=\bar{f}'|_{\cV}$ and $\lVert f-\bar{f}' \rVert_{\infty} < 2\lVert f \rVert_{\infty}$. 

The equality \eqref{eq:causality1} holds for each causal coupling $\pi_n $. It can be readily seen that 
\begin{align*}
\lim\limits_{n \to \infty} \int_{\cX \times \cY } g(y_{1:t})  f(x)\, \pi_n(dx, dy) &= \int_{\cX \times \cY} g(y_{1:t})  f(x)\, \pi(dx, dy) , \\
\lim\limits_{n \to \infty} \int_{\cX \times \cY} g(y_{1:t}) \bar{f}'(x_{1:t}) \, \pi_n(dx_{1:t}, dy_{1:t})&= \int_{\cX \times \cY} g(y_{1:t}) \bar{f}'(x_{1:t}) \, \pi(dx_{1:t}, dy_{1:t}),
\end{align*}
and 
\begin{align*}
\left| \int_{\cX \times \cY} g(y_{1:t}) \left(\bar{f}'(x_{1:t})-f(x_{1:t}) \right) \, \tilde\pi(dx_{1:t}, dy_{1:t})\right| \leq 2\delta \lVert f \rVert_{\infty} \lVert g \rVert_{\infty}, \, \forall \, \tilde\pi \text{ with $\cX$-marginal } \eta. 
\end{align*}
Therefore we conclude that 
\begin{align*}
&\left| \int_{\cX \times \cY} g(y_{1:t})  \left(\int_{\cX_{t+1:N}} f(x_{1:t}, x_{t+1:N}) \, \eta_{x_{1:t}}(dx_{t+1:N}) \right)\, \pi(dx_{1:t}, dy_{1:t}) \right. \\
& \quad \quad \left.  -\int_{\cX \times \cY} g(y_{1:t})  f(z)\, \pi(dx, dy) \right|\leq 4 \delta \lVert f \rVert_{\infty} \lVert g \rVert_{\infty}. 
\end{align*}
Letting $\delta \to 0$, we finish proving \eqref{eq:causality1}. 

Convexity of $\Pi_c(\eta, \cdot)$ is a direct consequence of \eqref{eq:causality1}. 

\end{proof}

Now we are ready to show our main result of this section. The precise assumption on the cost function $F$ is:

\begin{assume}\label{assume1} 
$ $
\begin{itemize}
\item[(i)]  $F: \cX \times \cY \times \cP(\cY)$ is  non-negative, $F(\cdot,\cdot,\nu)$ is continuous and bounded for each $\nu$, and $\nu\mapsto F(\cdot,\cdot,\nu)$ is continuous in supremum norm.
\item[(ii)] $\left\{y: \inf_{(x,{\nu}) \in \cX \times \cP(\cY)} F(x,y,{\nu}) \leq r  \right\}$ is compact for any $r>0$.
\item[(iii)] There exists a $y_0 \in \cY$ and $C<+\infty$ such that $$\sup_{\nu \in \cP(cY)} \int F (x,y_0,\nu) \, \eta(dx)\leq C.$$ 
\end{itemize}
\end{assume}

\begin{thm}\label{thm:1}
Under Assumption~\ref{assume1}, a solution to the fixed point problem \eqref{primalproblem} exists. 
\end{thm}

\begin{proof}
We show that the composition $T= Pj \circ \Phi$ has a fixed point. In \textit{Step 1}, we prove that $T(\nu)$ is relatively compact for any $\nu \in \cP(\cY)$, and hence we can restrict $T$ to a compact domain. In \textit{Step 2}, invoking Lemma~\ref{lem:causality}, we show that $T(\nu)$ is closed and convex. In \textit{Step 3} we prove the $T$ is upper-semicontinuous and therefore the existence of a fixed points for $T$ follows according to Lemma~\ref{lem:fixedpoint}.

\vspace{5pt} 
\textit{Step 1: } 
Take $y_0 \in \cX$ and $C<+\infty$ as in Assumption~\ref{assume1} (iii). It is clear that $\eta(dx)\delta_{y_0}(dy)\in \Pi_c(\eta, \cdot) $. Then for any putative $\pi \in \Phi({\nu})$ we would have  
\begin{align*}
\int_{\cX \times \cY} F(x,y,{\nu}) \, \pi(dx ,dy) \leq  \int_{\cX \times \cY} F(x,y_0,{\nu}) \, \eta(dx) \leq C.
\end{align*}
 From Assumption~\ref{assume1} (ii), we know that for any $ r >0$,  a compact subset $\cV_{r} \subset \cY$ exists  such that 
\begin{align*}
F(x,y,\nu) \geq r\, (\text{all $x,\nu$}) \quad \text{  whenever  } \quad y \not \in \cV_{r} . 
\end{align*}
Therefore we obtain the inequality 
\begin{align*}
Pj({\pi})[y \not \in \cV_r] \leq \pi[(x,y):F(x,y,\nu) \geq r]\leq \frac{\int_{\cX \times \cY} F(x,y,{\nu}) \, \pi(dx ,dy)}{r} \leq \frac{C}{r}.
\end{align*}
Define a subset $\cE \subset \cP(\cY)$ as 
\begin{align*}
\cE:=\left\{\nu \in \cP(\cY): \,  \nu[y \not \in \cV_r] \leq C/r, \, \forall r >0 \right\}.
\end{align*}
It is clear that $\cE$ is relatively compact, by Prokhorov theorem, as it is tight. By Portmanteau theorem, $\cE$ is also closed, since each set $\cY\backslash \cV_r$ is open. Hence $\cE$ is compact, and clearly convex too. By design we have $T({\nu}) \subset \cE$ for any ${\nu} \in \cP(\cY)$. We restrict the domain of $T$ to ${\cE}$, which is a compact and convex subset of the space of finite signed measures equipped with the weak topology.

\vspace{5pt}

\textit{Step 2: } We define $\Pi_c(\eta, \cE)$ as the subset of $\Pi_c(\eta, \cdot)$ consisting of measures with a $\cY$-marginal lying in $\cE$.
Note that $\Phi({\nu}) \subset \Pi_c(\eta, \cE)$, by Step 1. The compactness of $\cE$, Lemma~\ref{lem:causality}, and Prokhorov theorem, yield that $\Pi_c(\eta, \cE)$ is compact and so  $\Phi({\nu})$  is relatively compact. We notice that
\begin{align*}
\Phi(\nu)=\left\{\pi \in \Pi_c(\eta, \cE): \, \int F(x,y,\nu) \, \pi(dx ,dy) \leq \int F(x,y,\nu) \, \pi'(dx ,dy), \forall  \pi' \in \Pi_c(\eta, \cE) \right\},
\end{align*}
and by the compactness of $\Pi_c(\eta, \cE)$ and Assumption~\ref{assume1} (i) we obtain that $\Phi(\nu)$ is non-empty. By the same token, $\Phi(\nu)$ is closed and hence compact, and clearly $\Phi(\nu)$ is convex too. On the other hand, the map $Pj$ is continuous and linear. Hence $T({\nu})= Pj(\Phi({\nu}))$ is also nonempty, convex and compact. 
 
\vspace{5pt}

\textit{Step 3:} We prove that $T:\cE\to \cE$ is an upper-semicontinuous set-valued map. Thus there exists a fixed point in ${\cE}$, as a result of Lemma~\ref{lem:fixedpoint}. Since ${\cE}$ is compact, it is equivalent to show that the graph of $T$ is closed in ${\cE} \times {\cE}$. Take any sequence $({\nu}_n, \nu_n')_{n \geq 0} \subset \cE \times \cE$ such that 
$$\nu_n' \in T({\nu}_n),\quad {\nu}_n \to \h{\nu},\quad \nu_n' \to \h{\nu}'.$$
Let us prove that $\h{\nu}' \in T(\h{\nu})$. Note that for each $n$, there exists a $\pi_n \in \Phi({\nu}_n)$ such that $Pj (\pi_n)=\nu_n'$. Since $(\pi_n)_{n \geq 0} \subset \Pi_c(\eta, \cE)$, there exists a subsequence $(\pi_{n_k})_{k \geq 0}$ converging to $\h{\pi}$. According to Lemma~\ref{lem:causality}, we know that $\h{\pi} \in \Pi_c(\eta, \cdot)$ as well. It is clear then that $Pj(\h{\pi})=\h{\nu}'$. Let us verify that 
\begin{align}\label{eq:uppercontinuous}
 \int_{\cX \times \cY} F(x,y,\h{\nu}) \, \h{\pi}(dx ,dy) \leq  \int_{\cX \times \cY} F(x,y,\h{\nu}) \, \pi'(dx ,dy), \quad \forall  \pi' \in \Pi_c(\eta, \cdot).
\end{align}
According to the definition of $\pi_n \in \Phi({\nu}_n)$, we know that 
\begin{align*}
 \int_{\cX \times \cY} F(x,y,{\nu}_n) \, \pi_n(dx ,dy) \leq  \int_{\cX \times \cY} F(x,y,{\nu}_n) \, \pi'(dx ,dy), \quad \forall  \pi' \in \Pi_c(\eta, \cdot).
\end{align*}
Now using the uniform continuity of $F$ in Assumption~\ref{assume1} (i), and letting $n\to \infty$ in the above inequality, we conclude \eqref{eq:uppercontinuous}. 
\end{proof}

\begin{remark}
Inspection of the previous proof shows that Assumption~\ref{assume1} (i) could be weakened to 
\begin{itemize}
\item[(i')] The function $\nu\mapsto F(\cdot,\cdot,\nu)$ is continuous in sup-norm and for each $\nu$ the function $F(\cdot,\cdot,\nu)$ is bounded, jointly lower semicontinuous and continuous in its second argument.
\end{itemize}
As this seems to be a technicality, we do not develop this further.
\end{remark}

To guarantee the uniqueness of fixed point, we impose the following monotonicity condition on $F$.

\begin{assume}\label{assume2}
For any $\pi \in \Pi_c ( \eta, {\nu}), \pi' \in \Pi_c(\eta, {\nu}')$, {\color{red}if $\pi\neq \pi'$} then
\begin{align*}
\int_{\cX \times \cY} \left(F(x,y,{\nu})-F(x,y,{\nu}')\right) (\pi -\pi ' ) (dx, dy) >0. 
\end{align*}
\end{assume}

\begin{cor}
There exists at most one equilibrium under Assumption~\ref{assume2}. 
\end{cor}
\begin{proof}
Suppose there are two distinct equilibria $\pi \in \Pi_c ( \eta, \h{\nu}) $ and $\pi'\in \Pi_c ( \eta, \h{\nu}')$, so $\pi \in \Phi(\h{\nu})$ and $ \pi' \in \Phi(\h{\nu}')$. Then by definition
\begin{align*}
\int_{\cX \times \cY} F(x,y, \h{\nu}) \, \pi(dx, dy) &\leq \int_{\cX \times \cY} F(x,y, \h{\nu}) \, \pi'(dx, dy), \\
\int_{\cX \times \cY} F(x,y, \h{\nu}') \, \pi'(dx, dy) &\leq \int_{\cX \times \cY} F(x,y, \h{\nu}')\, \pi(dx, dy).
\end{align*}
Adding the above inequalities, we obtain that
\begin{align*}
\int_{\cX \times \cY} \left( F(x,y, \h{\nu})- F(x,y,\h{\nu}')\right)(\pi-\pi') (dx, dy)\leq 0,
\end{align*}
which contradicts Assumption~\ref{assume2}. 
\end{proof}

Here is a simple example of $F$ that satisfies Assumption~\ref{assume2}.
\begin{Ex} 
$F(x,y,{\nu})=c(x,y)+V[{\nu}](y)$, where $V$ is strictly Lasry-Lions monotone:
\begin{align*}
\int_{\cY} \left( V[\nu](y) - V[\nu'](y) \right) (\nu-\nu') (dy) > 0 \quad \text{for any $\nu \not = \nu'$.}
\end{align*}
\end{Ex}

\section{Fixed Point Iterations in the Quadratic Case}
In this section, we apply fixed point iterations / the contraction mapping theorem, in order to find the fixed point of \eqref{primalproblem}. As it is known, this is an algorithmic recipe unlike the result in Lemma~\ref{lem:fixedpoint}. Let us assume that $\cX_t=\cY_t=\R$, $t=1, \dotso, N$, and $$F(x,y,{\nu})= \frac{1}{2} \sum\limits_{t=1}^N |x_t-y_t|^2 + V[{\nu}](y),$$
where $y \mapsto \cV[{\nu}](y)$ is lower semicontimuous and bounded from below for any ${\nu} \in \cP(\cY)$. Due to the explicit structure of $F$, for any $\nu \in \cP(\cY)$ we can actually solve the minimization problem 
\begin{align}\label{eq:minimization}
\min_{\pi \in \Pi_c(\eta,\cdot)} \int_{\cX \times \cY} F(x,y,{\nu}) \, \pi(dx,dy) 
\end{align}
recursively. We first present the construction of minimizers of \eqref{eq:minimization}, and hence obtain a map $\Psi: \cP(\cY) \to \cP(\cY)$. Then we prove that $\Psi$ is actually a contraction under further assumptions. 

\subsection{Minimizer of \eqref{eq:minimization}} \label{subsec:dpp}

We first sketch the idea. For any $\eta \in \cP(\cX)$, define its disintegration
\begin{align*}
\eta_1(A)&:= \eta(A \times \R^{N-1}), \quad A \subset \R, \\
\eta^{x_{1:t}}&:= \cL^{\eta}( x_{t+1} \, | \, \cF^{\cX}_t), \quad \quad  \ \  t=1, \dotso, N-1. 
\end{align*}
Then we have that $\eta=\eta_1 \otimes \eta^{x_1} \otimes \dotso \otimes \eta^{x_{1:N-1}}$. Denote $V[{\nu}]_N(x,y):= V[{\nu}](y)$. For $t=N,\dotso, 1$, we define recursively
\begin{align}
\text{Opt}^{(x,y)_{1:t-1}}(x_t) & := \inf_{\bar{y} \in \cY_t} \left\{\frac{1}{2}|x_t-\bar{y}|^2+ V[{\nu}]_t(x_{1:t},y_{1:t-1},\bar{y})\right\} \label{eq:minimizer_Opt} \\
\label{eq:minimizer}
T[{\nu}]_{t}^{(x,y)_{1:t-1}}(x_t)&\in \cP \left ( \argmin_{\bar{y} \in \cY_t} \left\{\frac{1}{2}|x_t-\bar{y}|^2+ V[{\nu}]_t(x_{1:t},y_{1:t-1},\bar{y})\right\} \right),
\end{align}
and also
\begin{align}\label{eq:V_aux}
V[{\nu}]_{t-1}(x_{1:t-1}, y_{1:t-1})
&:=\int_{x_t \in \cX_t} \text{Opt}^{(x,y)_{1:t-1}}(x_t) \, \eta^{x_{1:t-1}}(dx_t), 
\end{align}
with the understanding that,  when $t=1$, we interpret $1:0=\emptyset $ and hence $ \eta^{x_{1:t-1}}:=\eta_1$ and so forth,  
in the above equation. We assume implicitly, for the time being, that the optimal value \eqref{eq:minimizer_Opt} depends measurably on the various parameters, and likewise that at least one optimizing kernel \eqref{eq:minimizer} exists. With each measurable choice of optimizing kernels in \eqref{eq:minimizer} it is possible to paste together a coupling as follows: by induction one defines first $\pi[\nu]_1\in \cP(\cX_1\times \cY_1)$  as $\eta_1(dx_1)T[\nu]_1^{\emptyset}(x_1)(dy_1)$ and then $\pi[\nu]^{(x,y)_{1:t-1}}(dx_t,dy_t):= \eta^{x_{1:t-1}}(dx_t)T[{\nu}]_{t}^{(x,y)_{1:t-1}}(x_t)(dy_t)$. Setting 
\begin{align}\label{eq:optimalcoupling}
\pi[\nu]:= \pi[\nu]_1 \otimes \pi[\nu]^{(x,y)_1} \otimes \dotso \otimes \pi[\nu]^{(x,y)_{1:N-1}},
\end{align}
we construct a causal coupling with $\cX$-marginal $\eta$. It can be proven that, given $\nu$, the set of all such couplings $\pi[\nu]$ is equal to $\Phi(\nu)$, i.e.\ the best responses to $\nu$. In particular $T(\nu)$, the set of $\cY$-marginals of best responses, is equal to the set of $\cY$-marginals of all such $\pi[\nu]$.

In the particular case that the selection \eqref{eq:minimizer} is a dirac measure (we still denote by $T[{\nu}]_{t}^{(x,y)_{1:t-1}}(x_t)$ the support of such dirac measure), then the above recipe allows us to build an adapted map $\cT[\nu](x)=(\cT[\nu]_1(x_1),\cT[\nu]_2(x_{1:2}),\dots,\cT[\nu]_N(x_{1:N}) )$ inductively as follows: $\cT[\nu]_1(x_1):= T[\nu]_1(x_1)$ and $\cT[\nu]_t(x_{1:t}):= T[{\nu}]_k^{\left(x_{1:k-1}, \cT[{\nu}]_{1:k-1}(x_{1:k-1})\right)}(x_k)$. Hence this defines a causal coupling with $\cX$-marginal $\eta$, supported on the graph of an adapted map, via $\pi[\nu]:=(id,\cT[\nu])(\eta)$.

%


\begin{prop}\label{prop:minimize} If \eqref{eq:minimizer_Opt} admits a minimizer (for any $ t=1, \dotso, N$, $x_{1:t} \in \cX_{1:t}$ and $y_{1:t-1} \in \cY_{1:t-1}$), then  $\pi[\nu]$ defined in \eqref{eq:optimalcoupling}
 minimizes \eqref{eq:minimization}. If \eqref{eq:minimizer_Opt} admits a unique minimizer (for any $ t=1, \dotso, N$, $x_{1:t} \in \cX_{1:t}$ and $y_{1:t-1} \in \cY_{1:t-1}$), then so does \eqref{eq:minimization} and its unique minimizer is supported on the graph of an adapted map.

\end{prop}

\begin{proof}

First of all we stress that the proposed construction of $\pi[\nu]$ is well-founded. This is proved by backwards induction from $t=N-1$ to $t=0$, and standard measurable selection arguments: Details aside, one applies \cite[Proposition 7.50]{BertsekasShreve} so that \eqref{eq:minimizer_Opt} is analytically measurable in its parameters, and \eqref{eq:minimizer} admits analytically measurable selectors. By the same token \eqref{eq:V_aux} is well-defined and analytically measurable. Then one iterates these arguments. The same arguments, applied to the case when \eqref{eq:minimizer_Opt} admits a unique minimizer (for any $ t=1, \dotso, N$, $x_{1:t} \in \cX_{1:t}$ and $y_{1:t-1} \in \cY_{1:t-1}$), show the well-foundedness of the mentioned coupling supported on the graph of an adapted map. Hence, it remains to discuss optimality.

Let $\gamma \in \Pi_c ( \eta, \cdot)$. Denote its disintegration by $\gamma_1 \otimes \gamma^{(x,y)_1} \otimes \dotso \otimes \gamma^{(x,y)_{1:N-1}}$. Since $\gamma$ is causal, the $\cX_t$-marginal of $\gamma^{(x,y)_{1:t-1}}$ is just $\eta^{x_{1:t-1}}$, and hence we have the disintegration $\gamma^{(x,y)_{1:t-1}}(dx_t,dy_t)=\eta^{x_{1:t-1}}(dx_t) \otimes \gamma^{(x,y)_{1:t-1}}(x_t,d y_t )$.

For any fixed $(x,y)_{1:N-1}$, according to our construction of $\pi$, it is clear that 
\begin{align*}
&\int_{\cX_N \times \cY_N} F(x,y,{\nu}) \,\gamma^{(x,y)_{1:N-1}}(dx_N, dy_N) \\
&= \frac{1}{2} \sum\limits_{t=1}^{N-1} |x_t-y_t|^2 + \int_{\cX_N \times \cY_N} \left(\frac{1}{2}|x_N-y_N|^2+ V[{\nu}]_N(x,y) \right) \,\gamma^{(x,y)_{1:N-1}}(dx_N, dy_N) \\
&= \frac{1}{2} \sum\limits_{t=1}^{N-1} |x_t-y_t|^2 \\
& \ \ \ + \int_{\cX_N } \eta^{x_{1:N-1}}(dx_N) \int_{\cY_N} \left(\frac{1}{2}|x_N-y_N|^2+ V[{\nu}]_N(x,y) \right) \,\gamma^{(x,y)_{1:N-1}}(x_N, dy_N) \\
&\geq \frac{1}{2} \sum\limits_{t=1}^{N-1} |x_t-y_t|^2+V[{\nu}](x_{1:N-1}, y_{1:N-1})\\
&= \int_{\cX_N \times \cY_N} F(x,y,{\nu}) \,\pi[\nu]^{(x,y)_{1:N-1}}(dx_N, dy_N), 
\end{align*}
since by definition $\pi[\nu]^{(x,y)_{1:N-1}}(x_N,dy_N)$ is concentrated on the set of minimizers of \eqref{eq:minimizer_Opt}. Similarly, for any fixed $(x,y)_{1:N-2}$, it can be readily seen that 
\begin{align*}
&\int_{\cX_{N-1:N} \times \cY_{N-1:N}} F(x,y,{\nu}) \, \gamma^{(x,y)_{1:N-2}}(dx_{N-1}, dy_{N-1}) \otimes \gamma^{(x,y)_{1:N-1}}(dx_{N}, dy_{N})  \\
& \geq \frac{1}{2} \sum\limits_{t=1}^{N-2} |x_t-y_t|^2  \\
& \ \ \ + \int_{\cX_{N-1} \times \cY_{N-1}} \left(\frac{1}{2} |x_{N-1}-y_{N-1}|^2+V[{\nu}]_{N-1}(x_{1:N-1}, y_{1:N-1})  \right) \gamma^{(x,y)_{1:N-2}}(x_{N-1}, dy_{N-1}) \\
& \geq \frac{1}{2} \sum\limits_{t=1}^{N-2} |x_t-y_t|^2+V[{\nu}]_{N-2}(x_{1:N-2}, y_{1:N-2}) \\
&=\int_{\cX_{N-1:N} \times \cY_{N-1:N}} F(x,y,{\nu}) \, \pi[\nu]^{(x,y)_{1:N-2}}(dx_{N-1}, dy_{N-1}) \otimes \pi[\nu]^{(x,y)_{1:N-1}}(dx_{N}, dy_{N}).
\end{align*}
Repeating the above argument iteratively for $t= N-2, \dotso, 1$, one can show that 
\begin{align*}
\int_{\cX \times \cY} F(x,y,{\nu}) \, (\gamma-\pi[\nu]) (dx,dy) \geq 0.
\end{align*}
\end{proof}

\subsection{$\cW_1$ contraction}

 As a first step, the convexity of $y_t \mapsto V[\nu]_t(x_{1:t},y_{1:t})$ will be analyzed in Proposition~\ref{prop:convex} under a convexity assumption on $V[\nu]$. As we also want to study contractivity of the best reply correspondence, we shall want to make our study of convexity quantitative. On its own this is not enough, and we shall also need a Lipschitz property of sorts. The precise assumptions needed here are:

\begin{assume}\label{assume3}
(i) For any ${\nu} \in \cP(\cY)$, $y \mapsto V[{\nu}](y)$ is twice continuously differentiable, and there exist two constants $\kappa \geq \lambda \geq 0$ such that $\kappa I_N \geq \nabla^2 V[{\nu}] \geq \lambda I_N$, and 
\begin{align}\label{eq:assume3}
\kappa+\lambda \geq 3 \times 5 \times  \dotso \times (2N-1) \times (\kappa-\lambda). 
\end{align}
(ii) There exists a constant $L>0$ such that $\nu \mapsto \nabla V [{\nu}](y)$ is $L$-Lipschitz for any $y \in \cY$.  
(iii) $\eta$ has finite first moment. 
\end{assume}

\begin{remark}
In Point (ii) of Assumption \ref{assume3}, the Lipschitz property is meant to hold under the 1-Wasserstein distance, defined by:
$$\cW(\mu,\nu):=\sup_{\substack{f:\mathbb R^N\to\mathbb R^N \\ 1-Lipschitz} }\int fd(\mu-\nu).$$

\end{remark}

For the convexity of $y_t \mapsto V[{\nu}]_{t}(x_{1:t}, y_{1:t})$, we need the following lemma whose proof is trivial and so it is omitted. 

\begin{lemma}\label{lem:diagonal} 
Suppose $M$ is a symmetric $N \times N$ matrix such that $\kappa \, I_N \geq M \geq \lambda \, I_N$. Then  
\begin{align*}
&M_{ii} \in [\lambda, \kappa], \quad   \quad i=1, \dotso, N; \\
&|M_{i,j}| \leq \sqrt{(M_{ii}-\lambda)(M_{jj}-\lambda) } \leq \kappa -\lambda, \quad 1 \leq i \not = j \leq N. 
\end{align*}
\end{lemma}

\begin{prop}\label{prop:convex} 
Under Points (i) and (iii) of Assumption~\ref{assume3}, the function 
$y_{1:k} \mapsto V[{\nu}]_k(x_{1:k},y_{1:k})$ is twice continuously differentiable, and $\kappa_k I_k \geq \nabla^2_{y_{1:k}} V[{\nu}]_k \geq \lambda_k I_k$, where 
\begin{align}\label{eq:lambda}
\lambda_k:= \frac{\kappa+\lambda-(2k+1)\dotso(2N-1)(k-\lambda)}{2},  \notag \\
\kappa_k:= \frac{\kappa+\lambda+(2k+1)\dotso(2N-1)(k-\lambda)}{2}.
\end{align}
\end{prop}
\begin{proof}
Suppose $t=N-1$. The minimization problem \eqref{eq:minimizer_Opt} is strictly convex for each value of $x$ and $y_{1:N-1}$. Hence the first order conditions of \eqref{eq:minimizer} completely characterize the unique minimizer $T[{\nu}]_{N}^{(x,y)_{1:N-1}}(x_N)$, and we obtain that
\begin{align}\label{eq:firstorder}
T[{\nu}]_{N}^{(x,y)_{1:N-1}}(x_N)+\pa_{y_N} V [{\nu}]_N\left( x,y_{1:N-1}, T[\h{\nu}]_{N}^{(x,y)_{1:N-1}}(x_N) \right)=x_N.
\end{align}

Let us show that $T[{\nu}]_{N}^{(x,y)_{1:N-1}}(x_N)$ is Lipschitz in $x_N$, which is necessary for us to exchange integral and derivative later in this argument. Denote $y_N= T[{\nu}]_{N}^{(x,y)_{1:N-1}}(x_N)$, $y_N' = T[{\nu}]_{N}^{(x,y)_{1:N-1}}(x_N')$. Due to the first order condition, we have that 
\begin{align*}
& (y_N-y_N')^2 + (y_N-y_N') \left( \pa_{y_N}V[\nu]_N(x,y_{1:N-1}, y_N)-\pa_{y_N}V[\nu]_N(x,y_{1:N-1}, y_N') \right) \\
&= (y_N-y_N')(x_N-x_N').
\end{align*}
According to Assumption~\ref{assume3} (i), the left hand side is bounded from below by $(1+\lambda) (y_N-y_N')^2$, and hence we obtain that 
\begin{align}\label{eq:Lipschitz11}
\left|T [{\nu}]_{N}^{(x,y)_{1:N-1}}(x_N')-T[{\nu}]_{N}^{(x,y)_{1:N-1}}(x_N') \right| =|y_N-y_N'| \leq \frac{|x_N-x_N'|}{1+\lambda}. 
\end{align}

As abbreviations, we take $T_N:=T[{\nu}]_{N}^{(x,y)_{1:N-1}}(x_N)$, $V_N=V[{\nu}]_N(x,y_{1:N-1}, T_N)$, and  
\begin{align*}
&V_{N-1}=V[{\nu}]_{N-1}(x_{1:N-1}, y_{1:N-1}) \\
&= \int_{x_N \in \cX_N} \frac{1}{2} |x- T[{\nu}]_{N}^{(x,y)_{1:N-1}}(x_N)|^2+ V[{\nu}]_N(x,y_{1:N-1},T[{\nu}]_{N}^{(x,y)_{1:N-1}}(x_N))  \, \eta^{x_{1:N-1}}(dx_N).
\end{align*}
According to the implicit function theorem,  which is applicable thanks to Assumption~\ref{assume3}(i), $T_N$ is continuously differentiable in $y$. By the envelope theorem, $V_{N-1}$ is  continuously differentiable (as $V_N$ is) in $y$, and we have 
\begin{align}\label{eq:induction}
\pa_{y_t} V_{N-1} =& \int_{x_N \in \cX_N} \left( (T_N -x_N) \pa_{y_t}T_N+\pa_{y_t}V_N + \pa_{y_{N}}V_N  \pa_{y_t}T_N\right) \, \eta^{x_{1:N-1}}(dx_N)  \notag \\
=&\int_{x_N \in \cX_N} \pa_{y_{t}}V_N  \, \eta^{x_{1:N-1}}(dx_N).
\end{align}
 We can deduce from \eqref{eq:Lipschitz11} and Lemma~\ref{lem:diagonal} that $\pa_{y_t} V[\nu]_N(x,y_{1:N-1}, T_N)$ is Lipschitz in $x_N$ and $y$, which justifies together with Assumption~\ref{assume3} (iii) the exchange of derivative and integral in \eqref{eq:induction}. By the same token, we deduce that $V_{N-1}$ is is effect twice continuously differentiable in $y$ and we have 
\begin{align*}
\pa^2_{y_k y_t} V_{N-1}=\int_{x_N \in \cX_N} \left( \pa^2_{y_k y_t } V_N + \pa^2_{y_t y_N} V_N \pa_{y_k}T_N  \right) \, \eta^{x_{1:N-1}}(dx_N).
\end{align*}
Taking derivative of \eqref{eq:firstorder} with respect to $y_k$, it can be seen that 
\begin{align*}
\pa_{y_k} T_N (1+\pa^2_{y_N}V_N)+\pa^2_{y_{k} y_N} V_N =0,
\end{align*}
and hence 
\begin{align*}
 \pa_{y_{k}}T_N=-\frac{\pa^2_{y_{k} y_N} V_N}{(1+\pa^2_{y_N}V_N)}.
\end{align*}
Therefore we obtain that 
\begin{align}\label{eq:aux_second_der_V}
\pa^2_{y_k y_t} V_{N-1} =&\int_{x_N \in \cX_N} \left( \pa^2_{y_t y_k} V_N -\frac{(\pa^2_{y_ty_N} V_N)(\pa^2_{y_{k} y_N} V_N)}{(1+\pa^2_{y_N}V_N)} \right) \, \eta^{x_{1:N-1}}(dx_N).
\end{align}

Take any vector $\xi=(\xi_1, \dotso, \xi_{N-1})$. Using \eqref{eq:aux_second_der_V}, Cauchy-Schwarz inequality, and Lemma~\ref{lem:diagonal}, it can be easily seen that 
\begin{align*}
\xi^\top \nabla_{y_{1:N-1}}^2 V_{N-1} \xi &\geq \lambda \lVert \xi \rVert^2 -\frac{(\sum_{j=1}^{N-1}\xi_j  \pa^2_{y_jy_N} V_N)^2}{1+\pa^2_{y_N} V_N} \\
&  \geq \left(\lambda -\sum_{j=1}^{N-1} \frac{(\pa^2_{y_j} V_N-\lambda)(\pa^2_{y_N} V_N-\lambda)}{1+\pa^2_{y_N} V_N} \right) \lVert \xi \rVert^2 \\
& \geq \left(\lambda -(N-1) (\kappa-\lambda) \right)\lVert \xi \rVert^2,
\end{align*}
and similarly 
\begin{align*}
\xi^\top \nabla_{y_{1:N-1}}^2 V_{N-1} \xi  \leq \left(\kappa +(N-1) (\kappa-\lambda) \right)\lVert \xi \rVert^2.
\end{align*}
Therefore, we obtain that 
\begin{align*}
(\kappa+(N-1)(\kappa-\lambda) )I_{N-1} \geq \nabla_{y_{1:N-1}}^2 V_{N-1}  \geq (\lambda -(N-1)(\kappa-\lambda))I_{N-1}, 
\end{align*}
or equivalently, that
\begin{align*}
\frac{(\kappa+\lambda +(2N-1)(\kappa-\lambda) )}{2}I_{N-1} \geq \nabla_{y_{1:N-1}}^2 V_{N-1}  \geq \frac{(\kappa+\lambda -(2N-1)(\kappa-\lambda) )}{2}I_{N-1}, 
\end{align*}

By induction, following the exact same arguments as above, we can get that for each $1\leq k \leq N-1$ the function $V_k$ is twice continuously differentiable in $y$ and 
\begin{align*}
\lambda_k I_k \leq \nabla_{y_{1:k}}^2 V_{k} \leq \kappa_k  I_k,
\end{align*}
where $\lambda_k, \kappa_k$ are defined as in \eqref{eq:lambda}. 
\end{proof}

By Proposition~\ref{prop:convex}, we know that $V[{\nu}]_t$ is convex in $y_t$ for any $t =1 ,\dotso, N$ under Assumption~\ref{assume3} (i). It follows that the problems \eqref{eq:minimizer_Opt} admit a unique minimizer. Then, by Proposition~\ref{prop:minimize}, it follows that Problem \eqref{eq:minimization} admits a unique minimizer $\pi[\nu]$. This minimizer is furthermore supported on the graph of an adapted map $\cT[\nu]$. To simplify notation, we write
\begin{align}\label{eq:psi}
\Psi: \cP(\cY) & \to \cP(\cY) \notag \\
{\nu} & \mapsto \cT[\nu](\eta)= Pj\circ \Phi(\nu), 
\end{align}
which is now an actual function, rather than a set-valued one. Observe that any minimizer of the problem 
\begin{align}\label{eq:causal}
\min_{\pi \in \Pi_c(\eta, \Psi({\nu}))} \int_{\cX \times \cY} F(x,y,{\nu}) \, \pi(dx,dy) {=  \min_{\pi \in \Pi_c(\eta, \Psi({\nu}))} \int_{\cX \times \cY}\frac{\|x-y\|^2}{2} \pi(dx,dy) + \int V[\nu](y)\Psi(\nu)(dy)\color{red}}. 
\end{align}
is also the minimizer of \eqref{eq:minimization}. Hence we conclude that $\pi[\nu]$ is also the unique minimizer of \eqref{eq:causal}.



Now we analyze the Lipschitz property of the function $({\nu},y) \mapsto T[{\nu}]_k^{(x,y)_{1:k-1}}(x_k)$, and after that we will show that $\Psi$ is a contraction under Assumption~\ref{assume3}. Here the contraction property is meant to hold under the 1-Wasserstein distance.

\begin{prop}\label{prop:Lip_estim}
Under Assumption~\ref{assume3},  it holds that 
\begin{align*}
 \left| T[{\nu}]_k^{(x,y)_{1:k-1}}(x_k)-T[{\nu}']_k^{(x,y')_{1:k-1}}(x_k) \right|  \leq \frac{L_k}{1+ \lambda_k} \cW_1({\nu}, {\nu}') +\frac{(\kappa_k -\lambda_k)\sum_{t=1}^{k-1} |y_t-y'_t|}{1+\lambda_k},
\end{align*}
where $L_N:=L$ and 
\begin{align}\label{eq:L}
L_k:=\frac{1+\kappa_{k+1}}{1+\lambda_{k+1}} L_{k+1}, \quad k=N-1, \dotso, 1.
\end{align}
\end{prop}
\begin{proof}
\textit{Step 1:}  First we prove that 
\begin{align}\label{eq:LIPT}
& \left| T[{\nu}]_{k}^{(x,y)_{1:k-1}}(x_k)- T[{\nu}']_k^{(x,y)_{1:k-1}}(x_k) \right|  \leq  \frac{L_k}{1+ \lambda_k} \cW_1({\nu}, {\nu}') .
\end{align}
Denote $\ol{y}_N=T[{\nu}]^{(x,y)_{1:N-1}}_N(x_N)$, $\ol{y}_N'=T[{\nu}']^{(x,y)_{1:N-1}}_N(x_N)$. It can be easily seen, by the first order optimality conditions as in \eqref{eq:firstorder}, that 
\begin{align*}
\ol{y}_N-\ol{y}_N'+\pa_{y_N}V[{\nu}]_N(x,y_{1:N-1}, \ol{y}_N) - \pa_{y_N}V[{\nu}']_N(x,y_{1:N-1},\ol{y}_N')=0,
\end{align*}
and hence 
\begin{align}\label{eq:Lip}
&(\ol{y}_N-\ol{y}_N')^2+(\ol{y}_N-\ol{y}_N')\left(\pa_{y_N}V[{\nu}]_N(x,y_{1:N-1}, \ol{y}_N) - \pa_{y_N}V[{\nu}]_N(x,y_{1:N-1},\ol{y}_N')\right) \notag \\
&=(\ol{y}_N-\ol{y}_N')\left(\pa_{y_N}V[{\nu}']_N(x,y_{1:N-1}, \ol{y}_N') - \pa_{y_N}V[{\nu}]_N(x,y_{1:N-1},\ol{y}_N')\right).
\end{align}
Using the convexity of $V[{\nu}]_N$ in $y_N$, the left hand side of \eqref{eq:Lip} is greater than 
\begin{align*}
(1+\lambda ) (\ol{y}_N-\ol{y}_N')^2,
\end{align*}
while the right hand side is smaller than $L|\ol{y}_N-\ol{y}_N'|\cW_1({\nu}, {\nu}')$. Therefore we obtain that 
\begin{align}\label{eq:LipY}
|\ol{y}_N-\ol{y}_N'| \leq \frac{L\cW_1({\nu}, {\nu}')}{1+\lambda}. 
\end{align}

According to \eqref{eq:induction}, we know that 

\begin{align*}
& \left| \nabla_{y_{1:N-1}} V[{\nu}]_{N-1} - \nabla_{y_{1:N-1}} V[{\nu}']_{N-1}\right| \\
&=\left| \int_{x_N \in \cX_N} \left(  \nabla_{y_{1:N-1}}V_N[\nu](x,y_{1:N-1}, \ol{y}_N)-\nabla_{y_{1:N-1}}V_N[\nu'](x,y_{1:N-1}, \ol{y}_N') \right)   \, \eta^{x_{1:N-1}}(dx_N) \right| \\
&\leq \left| \int_{x_N \in \cX_N} \left(  \nabla_{y_{1:N-1}}V_N[\nu](x,y_{1:N-1}, \ol{y}_N)-\nabla_{y_{1:N-1}}V_N[\nu'](x,y_{1:N-1}, \ol{y}_N) \right)   \, \eta^{x_{1:N-1}}(dx_N) \right| \\
& \ \ \ + \left| \int_{x_N \in \cX_N} \left(  \nabla_{y_{1:N-1}}V_N[\nu'](x,y_{1:N-1}, \ol{y}_N)-\nabla_{y_{1:N-1}}V_N[\nu'](x,y_{1:N-1}, \ol{y}_N') \right)   \, \eta^{x_{1:N-1}}(dx_N) \right|.
\end{align*}
The first term on the right hand side is bounded above by $L\cW_1({\nu}, {\nu}')$ due the point (ii) of Assumption~\ref{assume3}. By Lemma~\ref{lem:diagonal}, we obtain 
\begin{align*}
\left|  \pa_{y_N} \nabla_{y_{1:N-1}}V_N[\nu'](x,y_{1:N-1}, y_N) \right| \leq (\kappa-\lambda),
\end{align*}
and thus \eqref{eq:LipY} implies 
\begin{align*}
\left|  \nabla_{y_{1:N-1}}V_N[\nu'](x,y_{1:N-1}, \ol{y}_N)-\nabla_{y_{1:N-1}}V_N[\nu'](x,y_{1:N-1}, \ol{y}_N') \right| \leq \frac{ (\kappa -\lambda)L\cW_1({\nu}, {\nu}')}{1+\lambda}.
\end{align*}
Combining these estimates, we get that 
\begin{align*}
& \left| \nabla_{y_{1:N-1}} V[{\nu}]_{N-1} - \nabla_{y_{1:N-1}} V[{\nu}']_{N-1}\right| \\
& \leq L\cW_1({\nu}, {\nu}') + \frac{ (\kappa -\lambda)L\cW_1({\nu}, {\nu}')}{1+\lambda}.
\end{align*}
Recursively, we get that for $k=N-1, \dotso, 1$,
\begin{align*}
\left| \nabla V[{\nu}]_k- \nabla V[{\nu}']_k \right| \leq L_k \cW_1({\nu}, {\nu}'),
\end{align*}
and also \eqref{eq:LIPT}

\vspace{5pt}
\textit{Step 2: } Let us compute $|T[{\nu}]_k^{(x,y)_{1:k-1}}(x_k)-T[{\nu}]_k^{(x,y')_{1:k-1}}(x_k)|$. By first order condition, we have that 
\begin{align*}
&T[{\nu}]_k^{(x,y)_{1:k-1}}(x_k)-T[{\nu}]_k^{(x,y')_{1:k-1}}(x_k) \\
&+ \pa_{y_k} V[{\nu}]_k\left(x_{1:k}, y_{1:k-1}, T[{\nu}]_k^{(x,y)_{1:k-1}}(x_k)\right)- \pa_{y_k} V[{\nu}]_k\left(x_{1:k}, y'_{1:k-1}, T[{\nu}]_k^{(x,y')_{1:k-1}}(x_k)\right)=0.
\end{align*}
Similar to the derivation of \eqref{eq:LipY}, using Proposition~\ref{prop:convex} and Lemma~\ref{lem:diagonal} we get that 
\begin{align}\label{eq:LIPT2}
 \left| T[{\nu}]_k^{(x,y)_{1:k-1}}(x_k)-T[{\nu}]_k^{(x,y')_{1:k-1}}(x_k) \right| \leq \frac{(\kappa_k -\lambda_k)\sum_{t=1}^{k-1} |y_t-y'_t|}{1+\lambda_k}.
\end{align}

\vspace{5pt}
\textit{Step 3: }We combine the first two step using the triangle inequality.
\end{proof} 

\begin{prop}\label{thm:contraction}
Under Assumption~\ref{assume3} the function $\Psi$ defined in \eqref{eq:psi} is a contraction in $\cW_1$ metric if 
\begin{align}\label{eq:assumption}
\frac{L_1\left( \frac{\kappa_1-\lambda_1}{1+\lambda_1}\right)^N-L_1}{\kappa_1-2\lambda_1-1}<1. 
\end{align}

\end{prop}
\begin{proof}Let us recall the construction from Section \ref{subsec:dpp}: Using $T[{\nu}]_1, \dotso, T[{\nu}]_N$, we can  define $\cT[{\nu}]=(\cT[{\nu}]_1, \dotso, \cT[{\nu}]_N): \cX \to \cY$ inductively via 
\begin{align*}
 \cT[{\nu}]_1(x_1) &=T[{\nu}]_1(x_1), \\
 \cT[{\nu}]_k(x_{1:k}) &=T[{\nu}]_k^{\left(x_{1:k-1}, \cT[{\nu}]_{1:k-1}(x_{1:k-1})\right)}(x_k), \quad k=2, \dotso, N.
\end{align*}
It is clear that $\Psi({\nu})= (\cT[{\nu}]) ({\eta})$, and therefore 
\begin{align*}
\cW_1(\Psi({\nu}),\Psi({\nu}')) \leq \int_{x \in \cX}  \left| \cT[{\nu}](x)- \cT[{\nu}'](x) \right| \eta(dx).
\end{align*}

 Now according to Proposition \ref{prop:Lip_estim}, we have that 
\begin{align*}
|\cT[{\nu}]_1(x_1)- \cT[{\nu}']_1(x_1)|\leq \frac{L_1}{1+\lambda_1} \cW_1({\nu}, {\nu}'),
\end{align*}
and 
\begin{align*}
&|\cT[{\nu}]_2(x_{1:2})- \cT[{\nu}']_2(x_{1:2})| 
 =   \left|T[{\nu}]_2^{\left(x_1, \cT[{\nu}]_1(x_1)\right)}(x_2) -T[{\nu}']_2^{\left(x_1, \cT[{\nu}']_1(x_1)\right)}(x_2)\right|                          \\
&\leq \frac{L_2}{1+\lambda_2}\cW_1({\nu}, {\nu}') +\frac{\kappa_2-\lambda_2}{1+\lambda_2} |\cT[{\nu}]_1(x_{1})- \cT[{\nu}']_1(x_{1})|   \\
&\leq \frac{L_1}{1+\lambda_1} \left( 1+\frac{\kappa_1-\lambda_1}{1+\lambda_1}  \right) \cW_1({\nu}, {\nu}').
\end{align*}
By induction, one can prove that
\begin{align*}
& |\cT[{\nu}]_k(x_{1:k})- \cT[{\nu}']_k(x_{1:k})| \\
&\leq \frac{L_1}{1+\lambda_1} \left(1+ \dotso + \left(\frac{\kappa_1-\lambda_1}{1+\lambda_1} \right)^{k-1}  \right) \cW_1({\nu}, {\nu}'),
\end{align*}
and hence 
\begin{align*}
|\cT[{\nu}](x)- \cT[{\nu}'](x)| \leq \frac{L_1}{1+\lambda_1} \left(1+\dotso+\left(\frac{\kappa_1-\lambda_1}{1+\lambda_1} \right)^{N-1} \right)\cW_1({\nu}, {\nu}').
\end{align*}
Therefore  $\Psi$ is a contraction if \eqref{eq:assumption} is satisfied. 

\end{proof}

In the contracting case, it is well-known that there exist a unique fixed-point, which is furthermore determined by repeatedly iterating a map (fixed-point iterations). This tells us how to completely solve our equilibrium problem:

\begin{cor}
Under Assumption~\ref{assume3} and Condition \eqref{eq:assumption}, we have
\begin{enumerate}
\item The Cournot-Nash problem \eqref{eq:minimization} has a unique equilibrium $\pi$;
\item The second marginal of $\pi$ is the unique fixed point of $\Psi$, and it can be determined by the usual fixed-point iterations ``$\nu_{m+1}=\Psi(\nu_m)$''.
\item Conversely, after determining $\nu$ the unique fixed point of $\Psi$, the unique Cournot-Nash equilibrium $\pi$ is determined by minimizing \eqref{eq:causal} or equivalently by  taking $\pi=(id,T[\nu])(\eta)$ with $T[\nu]$ adapted and being uniquely ($\eta$-a.s.) determined via the recursions \eqref{eq:minimizer}.
\end{enumerate}
\end{cor}

\section{Application to Optimal Liquidation in a Price Impact Model}

We give a description of the price impact model in discrete time. An agent has at time 0 a number $Q_0>0$ of shares on a stock. At time 1, based on the available information, she aims to sell $y_1$ shares for their current price $S_1$, after which she is left with $Q_1=Q_0-y_1$ shares. This is iterated until time $N$, where she chooses to sell $y_{N}$ shares based on her current information, at the current price of $S_{N}$, leaving her with $Q_{N}=Q_{N-1}-y_{N}$ shares. 
The total earnings from this strategy is then 
$$E_N:= \sum_{i=1}^{N}y_iS_i.$$

As for the behaviour of the share prices $S_i$, we suppose that $S_0\in\R$ is known and that otherwise
$$S_{i}-S_{i-1}=x_i-x_{i-1}-m_i[\nu],$$
where $x\sim \eta$ is noise (wlog.\ we assume $x_0=0$) and $m_i[\nu]$ stands for the mean of the $i$-th marginal of a measure $\nu$. The idea is that the $i$-th marginal of $\nu$ is (in equilibrium) the distribution of the number of shares sold at time $i$, and so the term  $m_i[\nu]$ in the dynamics of $S$ indicates a permanent market impact caused by a population of identical, independent and negligible agents who at time $i$ decide to sell a number of shares. 

We define 
$$F(x,y,\nu):= AQ_N^2 + K\sum_{i=1}^{N} y_i^2-E_N,$$
where the first term accounts for a final cost of inventory and the second term models the accumulated transaction costs. Given a distribution $\nu$ of decisions taken by a population of agents, a negligible agent will aim to minimize the $\eta$-expectation of $F$ over the strategies adapted to the information of the share prices, or equivalently, the strategies adapted to $x$. More precisely, a pure equilibrium for this game would be an adapted map $\h{T}$ and a measure $\h{\nu}$ such that
\begin{enumerate}
\item[$(i)$] $\h{T}\in\argmin\limits_{T \, \text{adapted}}\int F(x, T(x),\h{\nu}) \, \eta(dx);$
\item[$(ii)$] $\h{T}(\eta)=\h{\nu}$.
\end{enumerate}

For this model we easily check that $F(x,y,\nu)=\frac{1}{2}\|x-y\|^2+V[\nu](y)$ where
\begin{align}\label{eq:Vfunction}
V[\nu](y):= \left(K-\frac{1}{2}\right)\sum_i y_i^2-S_0\sum y_i+A\left (Q_0- \sum_i y_i\right)^2+\sum_iy_i\sum_{k\leq i} m_k[\nu].
\end{align}
Hence $\nabla V[\nu](y)=(2K-1)y+\{2A(\sum y_i-Q_0)-S_0\}\mathbbm{1}_{N}+(\sum_{k\leq i}m_k[\nu] )_{i=1}^N$, and so $\nu \mapsto \nabla V[\nu](y)$ is $N$-Lipschitz with respect to the Wasserstein-1 distance, uniformly in $y$. Moreover, $\nabla^2V[\nu](y)=2A\mathbbm{1}_{N\times N}+(2K-1)I_N$, and so we have that $\kappa I_N \geq \nabla^2V[\nu](y) \geq \lambda I_N$, where $\kappa=2K-1+2AN$ and $\lambda=2K-1$. 
\begin{cor}\label{cor:priceimpact}
Take $L_N=N$, $\kappa=2K-1+2AN$, $\lambda=2K-1$, and define $L_t, \kappa_t, \lambda_t$, $t=N-1, \dotso 1$ recursively as in \eqref{eq:lambda} and \eqref{eq:L}. Then there exists a unique equilibrium if \eqref{eq:assume3} and \eqref{eq:assumption} are satisfied. 
\end{cor}

In our model, it can be readily seen that assumptions of Corollary~\ref{cor:priceimpact} are satisfied if $ N +A \ll K $. Now we show that it is not a potential game, and therefore cannot be covered by \cite{AcBaJi21}. Let us only prove for the simplest case $N=2$. 

\begin{lemma}
There exists no Fr\'{e}chet differentiable $\cE: \cP(\R^2) \to \R$ such that 
\begin{align}\label{eq:potential}
\lim\limits_{\e \to 0} \frac{\cE(\nu+\e \nu)-\cE(\nu) }{\e}=\int_{y \in \R^2} V[\nu](y) \, \mu(dy)
\end{align}
for any $\mu, \nu \in \cP(\R^2)$. 
\end{lemma}
\begin{proof}
Let us define 
\begin{align*}
\h{V}[\nu](y):= V[\nu](y)-  m_1[\nu]y_2,
\end{align*}
and 
\begin{align*}
\h{\cE}(\nu):=& \int_{y \in \R^2} \left(K-\frac{1}{2}\right)\sum_i y_i^2-S_0\sum y_i+A\left (Q_0- \sum_i y_i\right)^2 \, \nu(dy) \\
&+\frac{1}{2} (m_1[\nu])^2+ \frac{1}{2} (m_2[\nu])^2.
\end{align*}
It can easily verified that 
\begin{align*}
\lim\limits_{\e \to 0} \frac{\h{\cE}(\nu+\e \nu)-\cE(\nu) }{\e}=\int_{y \in \R^2} \h{V}[\nu](y) \, \mu(dy)
\end{align*}
for any $\mu, \nu \in \cP(\R^2)$. Therefore it suffices to show that $m_1[\nu]y_2$ is not potential. Otherwise suppose there exists some $\cE$ such that \eqref{eq:potential} holds with $V[\nu](y)=m_1[\nu]y_2$. 

Then it can be readily seen that  
\begin{align*}
& \cE(\delta_T \times \delta_1) - \cE(\delta_T \times \delta_0)\\
&= \int_0^1 \, dt \int \left( m_1[\delta_T \times \delta_0+t(\delta_T \times \delta_1-\delta_T \times \delta_0)]y_2\right) \, (\delta_T \times \delta_1-\delta_T \times \delta_0)(dy)=T, \\
& \cE(\delta_T \times \delta_1) - \cE(\delta_0 \times \delta_1)\\
&= \int_0^1 \, dt \int \left( m_1[\delta_0 \times \delta_1+t(\delta_T \times \delta_1-\delta_0 \times \delta_1)]y_2\right) \, (\delta_T \times \delta_1-\delta_0 \times \delta_1)(dy)=0, \\
& \cE(\delta_T \times \delta_0) - \cE(\delta_0 \times \delta_0)\\
&= \int_0^1 \, dt \int \left( m_1[\delta_0 \times \delta_0+t(\delta_T \times \delta_0-\delta_0 \times \delta_0)]y_2\right) \, (\delta_T \times \delta_0-\delta_0 \times \delta_0)(dy)=0, \\
& \cE(\delta_0 \times \delta_1) - \cE(\delta_0 \times \delta_0)\\
&= \int_0^1 \, dt \int \left( m_1[\delta_0 \times \delta_0+t(\delta_0 \times \delta_1-\delta_0 \times \delta_0)]y_2\right) \, (\delta_0 \times \delta_1-\delta_0 \times \delta_0)(dy)=0.
\end{align*}
Therefore we obtain that 
\begin{align*}
\cE(\delta_T \times \delta_1)-\cE(\delta_0 \times \delta_0)=& \cE(\delta_T \times \delta_1)-\cE(\delta_T \times \delta_0)+\cE(\delta_T \times \delta_0)-\cE(\delta_0 \times \delta_0)=T  \\
=& \cE(\delta_T \times \delta_1)-\cE(\delta_0 \times \delta_1)+\cE(\delta_0 \times \delta_1)-\cE(\delta_0 \times \delta_0)=0,
\end{align*}
which is a contradiction. 
\end{proof}

 To finish the article, let us present a simple example where we can illustrate how to compute the best response map $ \cT[\nu]$ and the fixed point $\nu$. 

\begin{Ex}
Suppose $N=2$ and $\eta=\frac{1}{2}(\delta_{0} +\delta_1) \times \frac{1}{2}(\delta_{0} + \delta_1)$. Take $F^{\epsilon}(x,y,\nu)= \frac{1}{2} \lVert x-y \rVert^2 + \epsilon V[\nu](y)$, where $V$ is given by \eqref{eq:Vfunction}. In the case of $\epsilon=1$, it is just price impact model above. Hence we know that $F^{\epsilon}$ is non-potential for $\epsilon >0$. Let us compute the best response given $\nu$: 
\begin{align*}
T^{\epsilon}[\nu]_2^{(x_1,y_1)}(x_2) =& \argmin_{\bar{y} \in \R} \bigg\{\frac{1}{2} |x_2-\bar{y}|^2+ \epsilon\left((K-1/2)(y_1^2+\bar{y}^2)-S_0(y_1 + \bar{y}) \right.  \\
 & \left. \quad \quad \quad \quad  +A(Q_0-y_1-\bar{y})^2+y_1m_1[\nu]+\bar{y} (m_1[\nu]+m_2[\nu])\right) \bigg\}  \\
 =& \frac{x_2+\epsilon(S_0-2A(y_1-Q_0)-m_1[\nu]-m_2[\nu])}{1+\epsilon(2K+2A-1)}. 
\end{align*}
Then $T^{\epsilon}[\nu]_1(x_1)$ is determined by the equations 
\begin{align*}
V^{\epsilon}[\nu]_1(x_1,y_1)=& \frac{1}{2}\left(\frac{1}{2} |T^{\epsilon}[\nu]_2^{(x_1,y_1)}(0)|^2+V[\nu](y_1, T^{\epsilon}[\nu]_2^{(x_1,y_1)}(0))  \right) \\
&+ \frac{1}{2}\left(\frac{1}{2} |1-T^{\epsilon}[\nu]_2^{(x_1,y_1)}(1)|^2+V[\nu](y_1, T^{\epsilon}[\nu]_2^{(x_1,y_1)}(1))  \right),
\end{align*}
and
\begin{align*}
0=T^{\epsilon}[\nu]_1(x_1)-x_1+\pa_{y_1} V^{\epsilon}[\nu]_1(x_1,T^{\epsilon}[\nu]_1(x_1)).
\end{align*}

After some computation, there exists some constants $a_1^{\e},\dotso, a_4^{\e}$, $\tilde{b}_1^{\e}, b_1^{\e}, \dotso, b_4^{\e}$ such that 
\begin{align*}
\cT^{\e}[\nu]_1(x_1)&:=T^{\e}[\nu]_1(x_1)= a_1^{\e} x_1+a_2^{\e} m_1[\nu]+a_3^{\e}m_2[\nu]+a_4^{\e}, \\
\cT^{\e}[\nu]_2(x_1,x_2)&:= T^{\epsilon}[\nu]_2^{(x_1,\cT^{\e}[\nu]_1(x_1))}(x_2)=b_1^{\e} x_2 + \tilde{b}_1^{\e} x_1 + b_2^{\e} m_1[\nu]+ b_3^{\e}m_2[\nu]+b_4^{\e}. 
\end{align*}
Since we assume that $\eta=\frac{1}{2}(\delta_{0} +\delta_1) \times \frac{1}{2}(\delta_{0} + \delta_1)$, the optimal response measure is given by
\begin{align*}
\hat{\nu}=\frac{1}{4} \sum_{x_1,x_2=0,1} \delta_{\left(\cT^{\e}[\nu]_1(x_1), \cT^{\e}[\nu]_2(x_1,x_2) \right)}, 
\end{align*}
and hence is completely determined by means $ m_1 [\nu]$ and  $m_2[\nu]$. Computing the means of $\hat \nu $, we obtain that 
\begin{align*}
m_1[\hat{\nu}]&=\frac{1}{2}a_1^{\e} +a_2^{\e} m_1[\nu]+a_3^{\e}m_2[\nu]+a_4^{\e} \\
m_2[\hat{\nu}]&=\frac{1}{2}b_1^{\e}  + \frac{1}{2} \tilde{b}_1^{\e}  + b_2^{\e} m_1[\nu]+ b_3^{\e}m_2[\nu]+b_4^{\e}.
\end{align*}
Therefore, the equilibrium is given by the solution of the linear system
\begin{align}\label{eq:linearsystem} 
m_1^{\e}&=\frac{1}{2}a_1^{\e} +a_2^{\e} m_1^{\e}+a_3^{\e}m_2^{\e}+a_4^{\e} \notag \\
m_2^{\e}&=\frac{1}{2}b_1^{\e}  + \frac{1}{2} \tilde{b}_1^{\e}  + b_2^{\e} m_1^{\e}+ b_3^{\e}m_2^{\e}+b_4^{\e}.
\end{align}
where variables $m_1^{\e}$, $m_2^{\e}$ stand for the mean of the first and second marginals of the equilibria.

It can be verified that if $F(x,y,\nu)$ satisfies assumptions of Theorem~\ref{thm:contraction}, then $F^{\e}(x,y,\nu)$ also satisfies that for any $ \e \in [0,1]$. 
Therefore, there always exists a unique solution of above linear equations \eqref{eq:linearsystem}. Although it is not immediate how to interpret this equilibrium, we do notice that as $\e \to 0$ the unique equilibria converge to the intuitive solution for $\e=0$. Indeed, as $\e \to 0$, we have $ a_1^{\e}, b_1^{\e} \to 1$ and $ a_2^{\e}, a_3^{\e},a_4^{\e}, \tilde{b}_1^{\e}, b_2^{\e}, b_3^{\e}, b_4^{\e} \to 0$. Therefore the fixed point $m_1^{\e}, m_2^{\e}$ both converge to $\frac{1}{2}$, and $\lim_{\e \to 0} \cT^{\e}[\nu]_1(x_1)=x_1$, $\lim_{\e \to 0} \cT^{\e}[\nu]_2(x_1,x_2)=x_2$. 
\end{Ex}

\bibliographystyle{siam}
\bibliography{biblio_CN2}
\end{document}